\documentclass[pdflatex,sn-mathphys-num]{sn-jnl}% Math and Physical Sciences Numbered Reference Style
\usepackage{graphicx}%
\usepackage{multirow}%
\usepackage{amsmath,amssymb,amsfonts}%
\usepackage{amsthm}%
\usepackage{mathrsfs}%
\usepackage{xcolor}%
\usepackage{textcomp}%
\usepackage{manyfoot}%
\theoremstyle{thmstyleone}%
\newtheorem{theorem}{Theorem}%  meant for continuous numbers
\newtheorem{corollary}[theorem]{Corollary}
\newtheorem{lemma}[theorem]{Lemma}
\newtheorem{proposition}[theorem]{Proposition}% 

\theoremstyle{thmstyletwo}%
\newtheorem{remark}{Remark}%

\theoremstyle{thmstylethree}%
\newtheorem{definition}{Definition}%

\begin{document}

\title[Completeness of the space of absolutely and upper integrable functions with values in a semi-normed space]{Completeness of the space of absolutely and upper integrable functions with values in a semi-normed space}


\author*[1,2]{\fnm{Rodolfo} \sur{Maza}}\email{remaza@mapua.edu.ph}\email{remaza@up.edu.ph}

\affil*[1]{\orgdiv{Department of Mathematics, School of Foundational Studies and Education}, \orgname{Map\'ua University}, \orgaddress{\street{Muralla Street}, \city{Manila}, \postcode{1002}, \country{Philippines}}}

\affil[2]{\orgdiv{Institute of Mathematics}, \orgname{University of the Philippines - Diliman}, \orgaddress{\street{C.P. Garcia Avenue}, \city{Quezon City}, \postcode{1101}, \country{Philippines}}}


\abstract{This paper studies absolute integrability for functions with values in semi-
	normed spaces and in locally convex topological vector spaces (LCTVS). We 
	introduce an \emph{upper-integral} approach (based on a $\rho$-variational 
	measure $\mu_{\rho}$) to define the spaces $\mathcal{U}^p_{\rho}$ of upper 
	integrable functions and investigate their functional-analytic properties.

	The main contributions are:
	\begin{itemize}
		\item the precise construction of the $\rho$-upper-integrability spaces 
		$\mathcal{U}^p_{\rho}(A;X)$ (and their Fr\'echet analogues), together with the 
		natural semi-norms $\|\cdot\|_{\mathcal{U}^p_{\rho}}$;
		\item measure-style inequalities adapted to the variational measure 
		$\mu_{\rho}$ (monotone continuity for ascending sets, Fatou-type lemma, and 
		Chebyshev inequality) within the $\rho$-upper-integral framework;
		\item functional-analytic results: sequential completeness of 
		$\mathcal{U}^p_{\rho}([a,b];X)$ when $X$ is sequentially complete (semi-normed 
		case), and sequential completeness of $\mathcal{U}^p([a,b];X)$ when $X$ is a 
		sequentially complete Fr\'echet space; and
		\item the closedness of the absolutely integrable subspace 
		$L^p_{\rho}([a,b];X)$ inside $\mathcal{U}^p_{\rho}([a,b];X)$ (hence 
		$L^p([a,b];X)$ is a closed Fr\'echet subspace of $\mathcal{U}^p([a,b];X)$ under 
		the usual hypotheses).
\end{itemize}
}




\maketitle

\section{Introduction}

This paper studies the absolute integrability of functions taking values in 
semi-normed spaces and locally convex topological vector spaces (LCTVS). We also
introduce an approach based on upper integrals that extends classical results 
from real-valued functions to the setting of LCTVS-valued functions.

Classical integrals such as Lebesgue, Bochner, and Pettis each face limitations when extended to semi-normed or general locally convex spaces \cite{DiestelUhl1977,MazaCanoy1,MazaCanoy2,MazaCanoy3,Rudin}. The Henstock-Kurzweil integral broadens the horizon, but still encounters restrictions in sequential completeness. The $\rho$-upper-integral introduced here addresses these gaps by combining measure-style flexibility with functional-analytic rigor.

We show that the space of absolutely integrable functions is a closed subspace 
within the space of upper integrable functions. Furthermore, we establish 
completeness results for the case of Fr\'echet spaces.

Absolute integrability coincides with Lebesgue integrability, which is a 
stronger notion than Kurzweil--Henstock integrability for real-valued functions 
(see, e.g., \cite{YeeVyborny}). A classical result in integration theory 
states that the space of Lebesgue integrable functions is complete and that 
simple functions provide a foundation for Lebesgue integration. It is natural to
ask whether an analogous result holds for LCTVS-valued functions. The work in 
\cite{Marraffa2} addressed this question for the variational McShane integral
by showing that it can be characterized as the completion of simple functions.

Our main objective is to explore fundamental properties of absolutely integrable
functions and to prove sequential completeness for the associated function 
spaces. To this end, we introduce the notion of an upper integral inspired by 
the $SL_\varphi$-integral in \cite{MazaCanoy4} and adapted from Thomson's 
definition for real-valued functions \cite{Thomson1,Thomson2,Thomson3},
using variational measures studied in \cite{MazaCanoy3}.

Throughout the paper, we follow standard notation and terminology in functional 
analysis and integration theory, referencing 
\cite{BogachevSmolyanov,MazaCanoy3,Rudin,Schaefer,YeeVyborny} for background material.

\section{Integration in Semi-Normed Spaces}

In this section, we introduce integration of functions on an interval with 
values in semi-normed spaces. Our treatment builds on results in 
\cite{MazaCanoy3} developed in the setting of locally convex topological vector 
spaces. Throughout this paper we consider functions
\[
f : \mathcal{I}[a,b] \times [a,b] \to X,
\]
where $\mathcal{I}[a,b]$ denotes the collection of non-degenerate closed 
subintervals of $[a,b]$. For the remainder of the paper we fix a semi-normed 
space $(X,\rho)$.

\begin{definition}[\(\delta\)-fine tagged partition]
	Let $\delta:[a,b]\to(0,\infty)$ be a gauge. A tagged partition
	$D=\{([u_i,v_i],t_i):1\le i\le n\}$ of $[a,b]$ is said to be 
	\emph{$\delta$--fine}
	if for each $i$ we have $t_i\in[u_i,v_i]$ and 
	$[u_i,v_i]\subseteq(t_i-\delta(t_i),t_i+\delta(t_i))$.
\end{definition}

\begin{definition}
	A function $f:[a,b]\to X$ is said to be \emph{$\rho$-SKH integrable} if 
	there exists a function $F:[a,b]\to X$ such that for every $\varepsilon>0$ there
	is a gauge $\delta$ on $[a,b]$ with the property that for every $\delta$-fine 
	tagged partition $D$ of $[a,b]$ one has
	\[
	(D)\sum \rho\!\big(F(v)-F(u)-f(t)(v-u)\big) < \varepsilon.
	\]
	We call $F$ a \emph{$\rho$-SKH primitive} of $f$. If $f \cdot 1_A$ is 
	$\rho$-SKH integrable for a subset $A\subseteq [a,b]$, we say that $f$ is 
	$\rho$-SKH integrable on $A$. The space of $\rho$-SKH integrable functions on 
	$A$ is denoted $SKH_{\rho}(A)$.
	
	Moreover, if $X$ is locally convex -- there is a separating family of 
	continuous semi-norms on \(X\), we say that $f:[a,b]\to X$ is \emph{strongly 
		Kurzweil--Henstock (SKH) integrable} if $f$ is $\rho$-SKH integrable for every 
	continuous semi-norm $\rho$ on $X$. In this case $F$ is called an SKH primitive 
	of $f$ and $F(b)-F(a)$ is the (unique) SKH integral of $f$.
\end{definition}

The notion of $\rho$-SKH integrability above coincides with the $\Phi_U$-SH 
integrability used in \cite{MazaCanoy3}, where $\Phi_U$ is the Minkowski 
functional of a balanced, convex neighbourhood $U$ of the origin. We use the 
semi-norm notation since our locally convex spaces are described via separating 
families of semi-norms.

Below we collect a few properties of SKH-integrable functions and address the 
non-uniqueness of the $\rho$-SKH primitive in the semi-normed setting.

An equivalent, more compact condition for integrability may be expressed using a
variant of function variation. From Thomson's works 
\cite{Thomson1,Thomson2,Thomson3}, the variational measure may be reformulated as
follows:

\begin{definition}
	Given a real-valued interval-point function 
	$h:\mathcal{I}[a,b]\times[a,b]\to \mathbb{R}$ and $E\subseteq[a,b]$,
	\begin{equation}
		\label{VariationalMeasure}
		\mu(h,E) \;=\; \inf_{\text{gauge }\delta} \; 
		\operatorname{Var}(h,E,\delta),
	\end{equation}
	where
	\[
	\operatorname{Var}(h,E,\delta)
	\;=\;
	\sup\Big\{ (D)\sum 1_E(t)\cdot (h)([u,v],t) :
	D\ \text{is }\delta\text{-fine} \Big\}.
	\]
\end{definition}

\begin{remark}
	\label{SKHIntegrability}
	Let $I$ be a compact interval. A function $f:I\to X$ is $\rho$-SKH 
	integrable if and only if there exists $F:[a,b]\to X$ such that
	\[
	\mu\big(\rho(\Delta F - \overline{f}),I\big)=0,
	\]
	where $\Delta F([u,v],t)=F(v)-F(u)$ and \(\overline{f}([u,v],t) = 
	f(t)(v-u)\).
\end{remark}

From the paper \cite{MazaCanoy3}, we have a useful variational measure for 
\(X\)-valued functions.

\begin{definition}
	The $\rho$-variation of a function $h:[a,b]\to X$ on a set 
	$E\subseteq[a,b]$ is defined by
	\[
	\label{rhoVariation}
	\mu^1_\rho(h,E) = \mu(\rho h,E).
	\]
\end{definition}

\begin{remark}
	\label{murho}
	The following facts follow from \cite{MazaCanoy3}: For any function 
	$h:[a,b]\to X$,
	\begin{enumerate}
		\item When $h([u,v],t)=t(v-u)$, the variational measure reduces 
		to the Lebesgue outer measure \(m^*\) restricted to $[a,b]$.
		\item For any $A,B\subseteq[a,b]$,
		\[
		\mu^1_\rho(h,A\cap B)
		= \mu^1_\rho\big(1_A\cdot h, B\big).
		\]
		\item The set function $E\mapsto \mu^1_\rho(h,E)$ is an outer 
		measure on $[a,b]$. We call it the $\rho$-variational measure induced by $h$.
		\item For fixed $E\subseteq[a,b]$, the mapping $f\mapsto 
		\mu^1_\rho(f,E)$ defines a seminorm on the vector space of functions \(f:[a,b] 
		\to X\) for which $\mu^1_\rho(f,E) < \infty$, that is, the following holds:
		\begin{enumerate}
			\item For any function \(f\) with \(\mu^1_\rho(f,E) < 
			\infty\) and real numbers \(\lambda\),
			\[\mu^1_\rho(\lambda f,E) \leq |\lambda| 
			\mu^1_\rho(f,E).\]
			\item For any function \(f\) and \(g\) with 
			\(\mu^1_\rho(f,E),\mu^1_\rho(g,E)<\infty\),
			\[\mu^1_\rho(f + g,E) \leq \mu^1_\rho(f,E) + 
			\mu^1_\rho(g,E).\]
		\end{enumerate}
	\end{enumerate}
\end{remark}

In Banach spaces (see, e.g., \cite{Gordon}) and in locally convex spaces (see 
\cite{PalugaPublishedDissertation}) the integral is unique. In a general semi-
normed space non-uniqueness may occur: if $F$ is a $\rho$-SKH primitive of $f$ 
and $H:[a,b]\to X$ is any nonzero, non-constant function with 
$\rho(\operatorname{im}(H))=0$, then $G(t)=F(t)+H(t)$ is also a primitive of 
$f$. Thus the integral is naturally interpreted as an equivalence class:
\[
\int_a^b f := \{\,F(b)-F(a): F \text{ is a } \rho\text{-SKH primitive of } f\,\}
= F(b)-F(a) + \ker\rho
\]
where $\ker\rho=\{x\in X:\rho(x)=0\}$.

One useful property of the $\rho$-variation for the \(X\)-valued function \(h\) 
is additivity over a finite non-overlapping cover of a closed interval.

\begin{theorem}
	\label{Additivityofvariationalmeasureoverintervals}
	Let $h:[a,b]\to X$ be a function. For any finite partition
	\[
	a=x_0 < x_1 < \cdots < x_n=b
	\]
	we have
	\[
	\mu^1_\rho(h,[a,b]) = \sum_{i=1}^n \mu^1_\rho(h,[x_{i-1},x_i]).
	\]
	Moreover, if $I_1,\dots,I_k$ are non-overlapping closed subintervals of 
	$[a,b]$, then
	\[
	\mu^1_\rho(h,[a,b]) \ge \sum_{j=1}^k \mu^1_\rho(h,I_j).
	\]
\end{theorem}

\begin{proof}
	For one direction, it suffices to show that for $a\le a_1<c<b_1\le b$,
	\begin{equation}
		\label{sup-additive}
		\mu^1_\rho(h,[a_1,b_1]) \ge \mu^1_\rho(h,[a_1,c]) + 
		\mu^1_\rho(h,[c,b_1]).
	\end{equation}
	The converse inequality is immediate from the monotonicity of \(\mu^1_\rho\).
	
	Fix $\varepsilon>0$. By the definition of $\mu$ there is a gauge 
	$\delta$ on $[a,b]$ such that every $\delta$-fine partial partition $D$ of 
	$[a,b]$ satisfies
	\[
	(D)\sum \rho h(t)(v-u) < \mu(\rho h,[a_1,b_1]) + \varepsilon.
	\]
	
	Moreover, we may assume that
	\[
	\begin{cases}
		(t-\delta(t),t+\delta(t)) \subset (a_1,c) & \text{for } t\in[a_1,c),\\[4pt]
		(t-\delta(t),t+\delta(t)) \subset (c,b_1) & \text{for } t\in(c,b_1],\\[4pt]
		\delta(a_1),\delta(c),\delta(b_1) < \varepsilon.
	\end{cases}
	\]
	Let $\delta$ be as above and let $\delta_1,\delta_2\le\delta$ be 
	arbitrary gauges.
	Choose any $\delta_1$--fine partial partition $D_1$ and any 
	$\delta_2$--fine partial
	partition $D_2$ of $[a,b]$.  Form $D_0$ by taking the union of all interval-points
	from $D_1\cup D_2$ except possibly those with tag equal to $a_1,c,b_1$; note that
	the interval-points in $D_1$ and $D_2$ are non-overlapping except possibly at 
	the points with tag $c$, hence $D_0$ is $\delta$--fine.  (Thus $D_0$ contains all 
	intervals of $D_1\cup D_2$ except the two that meet at $c$, which we keep separate.) 
	Hence
	\begin{align*}
		&(D_1)\sum \rho h(t)(v-u) + (D_2)\sum \rho h(t)(v-u)\\
		&\qquad\le (D_0)\sum \rho h(t)(v-u) + \rho(h(c))(v^1_c-u^1_c) + \rho(h(c))(v^2_c-u^2_c)\\
		&\qquad\quad + \rho(h(a_1))(v_{a_1}-u_{a_1}) + \rho(h(b_1))(v_{b_1}-u_{b_1})\\
		&\qquad< \mu^1_\rho(h,[a_1,b_1]) + \varepsilon + 2\rho(h(c))\varepsilon + \rho(h(a_1))\varepsilon + \rho(h(b_1))\varepsilon.
	\end{align*}
	Taking the supremum first over all $\delta_1$-fine $D_1$ (with $D_2$ 
	fixed) and then over all $\delta_2$-fine $D_2$, we obtain
	\[
	\mu^1_\rho(h,[a_1,c]) + \mu^1_\rho(h,[c,b_1])
	\le \mu^1_\rho(h,[a_1,b_1]) + C\varepsilon
	\]
	for some constant $C>0$ independent of $\varepsilon$. Letting 
	$\varepsilon\to 0$ yields \eqref{sup-additive}.
\end{proof}

We end this section by discussing the Kurzweil--Henstock integral of functions 
with values in $(X,\rho)$.

\begin{definition}
	A function $f:[a,b]\to X$ is \emph{$\rho$-Kurzweil--Henstock} 
	($\rho$-KH) integrable on $[a,b]$ if there exists a vector $z\in X$ such that 
	for every $\varepsilon>0$ there is a gauge $\delta$ on $[a,b]$ with the property
	that for every $\delta$-fine tagged partition $D$ of $[a,b]$ one has
	\[
	\rho \big((D)\sum f(t)(v-u) - z\big) < \varepsilon.
	\]
	The vector $z$ is then called an integral of $f$ over $[a,b]$.
\end{definition}

\begin{remark}
	The $\rho$-KH integral of $f$ is not necessarily unique: if $z$ is an 
	integral, then every element of $z+\ker\rho$ is also an integral, and the 
	difference of any two integrals lies in $\ker\rho$. Thus it is natural to define
	
	\[
	\int_a^b f := z + \ker\rho,
	\]
	the equivalence class of an integral $z$ of $f$.
	
	If $X$ is sequentially complete and $f$ is $\rho$-KH integrable on 
	$[a,b]$, then $f$ is $\rho$-KH integrable on every closed subinterval 
	$[u,v]\subseteq[a,b]$. Moreover the integral satisfies the usual additivity 
	property
	\begin{equation}
		\label{Additivityofintegral}
		\int_u^v f = \int_u^t f + \int_t^v f \qquad (a\le u < t < v \le 
		b).
	\end{equation}
\end{remark}

\begin{proposition}
	\label{IntegralEstimateforforHKIntegral}
	Let $f$ be $\rho$-KH integrable. Then
	\[
	\rho\!\left(\int_a^b f\right) \le \mu^1_\rho(f,[a,b]).
	\]
\end{proposition}

\begin{proof}
	Let $\varepsilon>0$. By integrability there exists a gauge $\delta$ and 
	a $\rho$-KH integral $z$ of $f$ such that for every $\delta$-fine tagged 
	partition $D$ one has
	\[
	\rho\!\big((D)\sum f(t)(v-u) - z\big) < \varepsilon.
	\]
	Choosing the same (or a sufficiently small) gauge $\delta$ we also have
	\[
	(D)\sum \rho f(t)(v-u) < \mu^1_\rho(f,[a,b]) + \varepsilon
	\]
	for every such $D$. Consequently
	\[
	\begin{aligned}
		\rho(z)
		&\le \rho\!\big((D)\sum f(t)(v-u)\big) + \varepsilon \\
		&\le \big((D)\sum \rho f(t)(v-u)\big) + \varepsilon \\
		&< \mu^1_\rho(f,[a,b]) + 2\varepsilon.
	\end{aligned}
	\]
	Letting $\varepsilon\to 0$ yields the desired estimate.
\end{proof}

\section{Functions with upper integrable seminorm}

For real-valued functions on a compact interval, every Lebesgue integrable
function is Kurzweil--Henstock (KH) integrable, and the two integrals coincide
\cite{YeeVyborny,Henstock}. Moreover, a measurable function is Lebesgue
integrable if and only if both the function and its absolute value are KH 
integrable.
This equivalence motivates definitions of absolute KH-integrability for 
functions
with values in a semi-normed space.

Inspired by that connection, in this section we develop the concept of absolute 
SKH-integrability for functions taking values in a fixed semi-normed space 
\((X,\rho)\). We begin with the notion used in \cite{MazaCanoy3}.

\begin{definition}
	A function \(f:[a,b]\to X\) is said to be \emph{absolutely 
		\(\rho\)-Kurzweil--Henstock (KH) integrable} if \(f\) is \(\rho\)-SKH integrable
	and the real-valued function \(\rho(f)\) is Kurzweil--Henstock integrable.
\end{definition}

For an absolutely KH-integrable \(f\) the integral of \(\rho(f)\) satisfies
\begin{equation}
	\label{pAbsoluteIntegral}
	\int_a^b \rho(f) \;=\; \mu^1_\rho(f,[a,b])
	\;=\; \inf_{\text{gauge }\delta}\sup\Big\{(D)\sum (\rho f)(t) (v-u): 
	D\text{ is }\delta\text{-fine}\Big\}.
\end{equation}
Indeed, fix \(\varepsilon>0\). There exists a gauge \(\delta\) on \([a,b]\) such
that for every \(\delta\)-fine partition \(D\) one has
\[
|(D)\sum (\rho f)(t) (v-u) - \int_a^b \rho(f)| < \varepsilon,
\]
hence \((D)\sum (\rho f)(t) (v-u) \le \int_a^b \rho(f) + \varepsilon\), which 
gives \(\mu^1_\rho(f,[a,b]) \le \int_a^b \rho(f) + \varepsilon\); letting 
\(\varepsilon\to 0\) yields one inequality, and the opposite inequality is 
proved similarly.

The \(\rho\)-SKH primitives of absolutely integrable functions satisfy natural 
absolute-continuity and bounded-variation properties. The next estimate is the 
analogue of Proposition~\ref{IntegralEstimateforforHKIntegral}.

\begin{proposition}[\cite{MazaCanoy3}]
	\label{IntegralEstimateforp=1}
	Let \(f\) be \(L_\rho^1\)-integrable with \(L_\rho^1\)-primitive \(F\). 
	Then
	\[
	\rho\!\left(\int_a^b f\right) \le \int_a^b \rho(f).
	\]
\end{proposition}

We now generalize the notion of absolute integrability.

\begin{definition}
	Let \(1 \le p < \infty\). A function \(f:[a,b]\to X\) is said to be 
	\(L_\rho^p\)-integrable if \(f\) is \(\rho\)-SKH integrable and the real-valued 
	function \(\rho(f)^p\) is Kurzweil--Henstock integrable. Denote by 
	\(L_\rho^p([a,b];X)\) the space of such \(X\)-valued functions on \([a,b]\).
	
	When \(X\) is locally convex, \(f:A\to X\) is called \(L^p\)-integrable 
	if \(f\) is \(L^p_\rho\)-integrable for every continuous semi-norm \(\rho\); the
	corresponding space is denoted \(L^p([a,b];X)\).
\end{definition}

From Remark \ref{murho}, we observe that omitting the integrability condition on
the function results to a larger space of functions.

\begin{definition}
	Let \(1 \le p < \infty\) and \(E \subseteq [a,b]\). For a real-valued 
	function \(h\) on \([a,b]\), let
	\[\operatorname{Var}^p(h,E,\delta) = \sup\Big\{(D)\sum h^p(t) (v-u): 
	D\text{ is }\delta\text{-fine}\Big\}
	\]
	and
	\[
	\mu^p_\rho(f,E) := \inf_{\text{gauge }\delta} \operatorname{Var}^p(\rho 
	f,E,\delta).
	\]
	
	A function \(f:[a,b]\to X\) is \(\mathcal{U}^p_\rho\)-integrable on 
	\([a,b]\) if
	\[
	\mu^p_\rho(f,[a,b]) < \infty.
	\]
	For a set \(A\subseteq[a,b]\) we say \(f\) is 
	\(\mathcal{U}^p_\rho\)-integrable on \(A\) if \(f1_A\) is 
	\(\mathcal{U}^p_\rho\)-integrable on \([a,b]\).
	
	Define
	\[
	\mathcal{U}^p_\rho(A;X)=\{\,f: \mu^p_\rho(f,A)<\infty\,\} \quad 
	\text{and}
	\quad
	\|f\|_{\mathcal{U}^p_\rho(A;X)} := \big(\mu^p_\rho(f,A)\big)^{1/p}.
	\]
	If \(X\) is locally convex, write \(\mathcal{U}^p(A;X)\) for the class 
	of functions that are \(\mathcal{U}^p_\rho\)-integrable for every continuous 
	semi-norm \(\rho\).
\end{definition}

\begin{remark}
	\ 
	\begin{enumerate}
		\item The notation \(\mathcal{U}^p_\rho\) is motivated by 
		``upper integrability'' in \cite{TulceaTulcea} and the viewpoint in 
		\cite{Thomson1}.
		\item One easily checks that for a semi-normed space $(X,\rho)$ 
		one has
		\[
		L^p_\rho(A;X)=\mathcal U^p_\rho(A;X)\cap SKH_\rho(A),
		\]
		and when $X$ is locally convex (with continuous semi-norms $\rho$), 
		the analogous equality holds for $L^p(A;X)$ and $\mathcal U^p(A;X)$ (i.e. 
		the equality holds for each $\rho$ and therefore in the locally convex topology).
	\end{enumerate}
\end{remark}

Many classical properties of Lebesgue integrable functions have analogues in 
this setting.

\begin{lemma}\label{lem:Holder-variational}
	Let \(1<p<\infty\) and let \(p^{*}\) be the conjugate exponent, 
	\(1/p+1/p^{*}=1\).
	Assume \(f\in \mathcal{U}^{p}_{\rho}(A;X)\) with 
	\(\|f\|_{\mathcal{U}^{p}_{\rho}(A;X)}\neq 0\). Define
	\[
	f^{*}(x):=\|f\|_{\mathcal{U}^{p}_{\rho}(A;X)}^{\,1-p}\,(\rho(f(x)))^{p-1},
	\qquad x\in A.
	\]
	Then \(f^{*}\in \mathcal{U}^{p^{*}}(A;\mathbb{R})\) with 
	\(\|f^{*}\|_{\mathcal{U}^{p^{*}}(A;\mathbb{R})}=1\), and
	\[
	\|f\|_{\mathcal{U}^{p}_{\rho}(A;X)} = \mu^1_\rho\big(f^{*}f,A\big).
	\]
	Moreover, if \(g\in \mathcal{U}^{p^{*}}(A;\mathbb{R})\) then \(g f \in 
	L\mathcal{U}^{1}_{\rho}(A;X)\) and
	\begin{equation}
		\label{eq:Holder-variational}
		\mu^1_\rho\big(gf,A\big) \le 
		\|g\|_{\mathcal{U}^{p^{*}}(A;\mathbb{R})}\,\|f\|_{\mathcal{U}^{p}_{\rho}(A;X)}.
	\end{equation}
\end{lemma}

\begin{proof}
	Since \(\|f\|_{\mathcal{U}^{p}_{\rho}}>0\), then pointwise
	\[
	\big(f^{*}(x)\big)^{p^{*}}
	= \|f\|_{\mathcal{U}^{p}_{\rho}}^{-p}\,(\rho(f(x)))^{p},
	\]
	since \((p-1)p^{*}=p\) and \(p^{*}(1-p)=-p\). By positive homogeneity of
	\(\mu(\cdot,A)\),
	\[
	\mu\big((f^{*})^{p^{*}},A\big)
	= \|f\|_{\mathcal{U}^{p}_{\rho}}^{-p}\,\mu\big(\rho(f)^{p},A\big)
	= 1,
	\]
	so \(f^{*}\in \mathcal{U}^{p^{*}}(A;\mathbb{R})\) and 
	\(\|f^{*}\|_{\mathcal{U}^{p^{*}}}=1\).
	
	Moreover,
	\[
	\rho\big(f^{*}(x)f(x)\big)
	= f^{*}(x)\,\rho(f(x))
	= \|f\|_{\mathcal{U}^{p}_{\rho}}^{\,1-p}\,\rho(f(x))^{p},
	\]
	whence, by homogeneity,
	\[
	\mu^1_\rho\big(f^{*}f,A\big)
	= \|f\|_{\mathcal{U}^{p}_{\rho}}^{\,1-p}\,\mu\big(\rho(f)^{p},A\big)
	= \|f\|_{\mathcal{U}^{p}_{\rho}}.
	\]
	
	For the second part, let \(p^{*}\) be the conjugate exponent 
	(\(1/p+1/p^{*}=1\)) and assume \(\|f\|_{\mathcal{U}^p_{\rho}(A;X)} = 1\) and 
	\(\|g\|_{\mathcal{U}^{p^{*}}(A;\mathbb{R})} = 1\). By Young's inequality, for 
	each \(x\in A\) we have
	\[
	\rho(f(x))\,|g(x)|
	\le \frac{1}{p}\,\rho(f(x))^{p} + \frac{1}{p^{*}}\,|g(x)|^{p^{*}}.
	\]
	Multiply both sides by the length of a subinterval \((v-u)\), sum over 
	the tagged intervals of a \(\delta\)-fine partition \(D\), take the supremum over 
	\(\delta\)-fine partitions and then the infimum over gauges to obtain
	\[
	\mu^1_\rho\big(gf,A\big)
	\le \frac{1}{p}\,\mu^p_\rho\big(f,A\big) + 
	\frac{1}{p^{*}}\,\mu^{p^*}_{|\cdot|} (g,A).
	\]
	Observe that
	\[
	\mu^p_\rho\big(f,A\big)=\|f\|^p_{\mathcal{U}^p_{\rho}(A;X)} = 1,
	\qquad
	\mu^{p^*}_{|\cdot|}\big(g,A\big)=\|g\|^{p^*}_{\mathcal{U}^{p^{*}}(A;\mathbb{R})}=1.
	\]
	Hence the right-hand side equals \(1/p+1/p^{*}=1\), so
	\[
	\mu^1_\rho\big(gf,A\big) \le 1.
	\]
	In general, if \(\|f\|_{\mathcal{U}^p_{\rho}(A;X)}\neq 0\) and 
	\(\|g\|_{\mathcal{U}^{p^{*}}(A;\mathbb{R})}\neq 0\) then the functions below 
	have corresponding semi-norms equal to \(1\):
	\[
	\|f\|_{\mathcal{U}^p_{\rho}(A;X)}^{-1} f,\qquad 
	\|g\|_{\mathcal{U}^{p^{*}}(A;\mathbb{R})}^{-1} g,
	\]
	It follows from the homogeneity of \(\mu\) that
	\begin{align*}
		\mu^1_\rho\big(gf,A\big)
		&= 
		\|g\|_{\mathcal{U}^{p^{*}}(A;\mathbb{R})}\,\|f\|_{\mathcal{U}^p_{\rho}(A;X)} 
		\mu^1_\rho \big(\|g\|^{-1}_{\mathcal{U}^{p^{*}}(A;\mathbb{R})}g 
		\|f\|^{-1}_{\mathcal{U}^p_{\rho}(A;X)}f,A\big)\\
		&\le 
		\|g\|_{\mathcal{U}^{p^{*}}(A;\mathbb{R})}\,\|f\|_{\mathcal{U}^p_{\rho}(A;X)}
	\end{align*}
	which is exactly \eqref{eq:Holder-variational}.
\end{proof}

\begin{lemma}
	\label{lem:Up-seminorm-and-embedding}
	Let \(A\subseteq[a,b]\) and \(1\le p<\infty\). The space 
	\(\mathcal{U}^{p}_{\rho}(A;X)\) is a seminormed space with seminorm
	\[
	\|f\|_{\mathcal{U}^{p}_{\rho}(A;X)} := \mu\big(\rho(f)^{p},A\big)^{1/p}.
	\]
	Moreover, if \(1\le p<q\) then \(\mathcal{U}^{q}_{\rho}(A;X)\subset 
	\mathcal{U}^{p}_{\rho}(A;X)\) and
	\begin{equation}\label{eq:Up-embedding}
		\|f\|_{\mathcal{U}^{p}_{\rho}(A;X)} \le 
		m(A)^{\,\frac{1}{p}-\frac{1}{q}} \,\|f\|_{\mathcal{U}^{q}_{\rho}(A;X)},
	\end{equation}
	where \(m(A)\) denotes the Lebesgue measure of \(A\).
\end{lemma}

\begin{proof}
	Consider \(f,g \in \mathcal{U}^p_{\rho}(A;X)\) and let \(h = (f+g)^* \in 
	\mathcal{U}^{p^*}(A;\mathbb{R})\). Then
	\[\|(f+g)^*\|_{\mathcal{U}^q_{\rho}(A;X)} = 1.\]
	By Lemma \ref{lem:Holder-variational},
	\begin{align*}
		\|f+g\|_{\mathcal{U}^p_{\rho}(A;X)} &= \mu(h(f+g),A)\\
		&\le \|h\|_{\mathcal{U}^{p^*}(A;\mathbb{R})} 
		\|f\|_{\mathcal{U}^p_{\rho}(A;X)} +\|h\|_{\mathcal{U}^{p^*}(A;\mathbb{R})} 
		\|g\|_{\mathcal{U}^p_{\rho}(A;X)}\\
		&= \|f\|_{\mathcal{U}^p_{\rho}(A;X)} + 
		\|g\|_{\mathcal{U}^p_{\rho}(A;X)}.
	\end{align*}
	For the second part, let \(f \in \mathcal{U}_\rho^q(A;X)\). Since 
	\(\rho(f)^q = (\rho(f)^p)^{q/p}\), it follows that \(\rho(f)^p \in 
	\mathcal{U}^{\frac{q}{p}}(A;\mathbb{R})\). The inequality 
	\eqref{eq:Holder-variational} of Lemma \ref{lem:Holder-variational} for \(X = 
	\mathbb{R}\) with \(\rho = | \cdot|\) is given by
	\begin{equation}
		\label{eq:RealHolder-variational}
		\mu^1_{| \cdot |}\big(gf,A\big) \le \|g\|_{\mathcal{U}^{p^{*}}(A
			;\mathbb{R})}\,\|f\|_{\mathcal{U}^{p}_{\rho}(A;\mathbb{R})}.
	\end{equation}
	
	Applying \eqref{eq:RealHolder-variational} to the functions 
	\(\rho(f)^p\) and \(1_A\) with exponent \(\frac{q}{p}\),
	\begin{equation}
		\label{eq:RealHolder-variationalApplied1}
		\mu^1_{|\cdot|}(\rho (f)^p,A) \leq 
		\|1_A\|_{\mathcal{U}^{(\frac{q}{p})^*}(A;\mathbb{R})} 
		\|\rho(f)^p\|_{\mathcal{U}^{\frac{q}{p}}(A;\mathbb{R})}
	\end{equation}
	Now, from the following equations
	\[\mu^1_{|\cdot|}(\rho (f)^p,A) = \inf_{\text{gauge }\delta} 
	\operatorname{Var}^p(|\rho (f)^p|,A,\delta) = \mu^p_\rho(f,A)\]
	and
	\[
	\begin{gathered}                
		\|\rho(f)^p\|^{\frac{q}{p}}_{\mathcal{U}^{\frac{q}{p}}(A;\mathbb{R})} = 
		\mu^{q/p}_{|\cdot|}(\rho (f)^p,A) = \inf_{\text{gauge }\delta} 
		\operatorname{Var}^{q/p}(|\rho (f)^p|,A,\delta)\\
		=  \inf_{\text{gauge }\delta} \operatorname{Var}^q(\rho 
		(f),A,\delta) = \mu^q_\rho(f,A) = \|f\|^q_{\mathcal{U}^q_\rho(A;X)},
	\end{gathered}
	\]
	inequality \eqref{eq:RealHolder-variationalApplied1} turns into
	\[\mu^p_\rho(f,A) \leq 
	\|1_A\|_{\mathcal{U}^{(\frac{q}{p})^*}(A;\mathbb{R})}  
	\|f\|^p_{\mathcal{U}^q_\rho(A;X)}.\]
	Hence,
	\[\|f\|_{\mathcal{U}_\rho^p(A;X)} \leq 
	m(A)^{\frac{1}{p}-\frac{1}{q}} \|f\|_{\mathcal{U}_\rho^q(A;X)}.\]
\end{proof}

\begin{theorem}
	\label{measureofascendingsets}
	Let \(f : [a,b] \rightarrow X\) and let \(\{A_n\}\) be an ascending 
	sequence of subsets of \([a,b]\); set \(A=\bigcup_{n} A_n\). Then
	\begin{equation}
		\label{MeasureofAscendingSets}
		\mu^p_\rho(f,A) = \lim_{n \rightarrow \infty} \mu^p_\rho(f,A_n).
	\end{equation}
	In particular, if \(\{\mu^p_\rho(f,A_n)\}\) is bounded above then \(f 
	\in \mathcal{U}^p_{\rho}(A;X)\).
\end{theorem}

\begin{proof}
	Fix \(\varepsilon > 0\).  
	Define the mutually disjoint sequence of sets
	\[
	E_1 = A_1,\quad 
	E_2 = A_2 \setminus A_1,\quad 
	E_3 = A_3 \setminus (A_2 \cup A_1),\ \ldots
	\]
	so that for each \(N \geq 1\), we have
	\[
	A_N = \bigcup_{j=1}^N E_j.
	\]
	
	For each \(n \geq 1\), choose a gauge \(\delta_n\) on \([a,b]\) such 
	that
	\[
	Var^p (\rho f,E_n,\delta_n) \;\le\; \mu^p_\rho (f,E_n) + 
	\frac{\varepsilon}{2^n}.
	\]
	Define a gauge \(\delta\) on \([a,b]\) by setting \(\delta|_{E_n} = 
	\delta_n\) for each \(n\).  
	
	Let \(D = \{([u_i,v_i],t_i)\}_{i=1}^m\) be a \(\delta\)-fine partial 
	partition of \([a,b]\).  
	Since the tags of \(D\) that lie in \(A\) must all lie in some finite 
	stage \(A_N\), we have
	\[
	1_A(t_i) = \sum_{j=1}^N 1_{E_j}(t_i).
	\]
	
	Therefore,
	\[
	\begin{aligned}
		\sum_{i=1}^m &\rho(f(t_i))^p 1_A(t_i)(v_i-u_i)\\
		&= \sum_{j=1}^N \sum_{i=1}^m \rho(f(t_i))^p 1_{E_j}(t_i)(v_i-u_i) 
		&& \text{(since \(A_N = \bigcup_{j=1}^N E_j\))} \\[6pt]
		&\leq \sum_{j=1}^N Var^p(1_{E_j}\cdot \rho f,[u_i,v_i],\delta) 
		&& \text{(definition of variation as a supremum)} \\[6pt]
		&\leq \sum_{j=1}^N \Big( \mu^p_\rho(1_{E_j}\cdot f,[u_i,v_i]) + 
		\tfrac{\varepsilon}{2^j} \Big) 
		&& \text{(by the choice of \(\delta_j\))} \\[6pt]
		&\leq \sum_{j=1}^N \Big( \mu^p_\rho(1_{A_N}\cdot f,[u_i,v_i]) + 
		\tfrac{\varepsilon}{2^j} \Big) 
		&& (\text{monotonicity of } \mu^p_\rho)\\[6pt]
		&= \mu^p_\rho(f, A_N) + \sum_{j=1}^N \tfrac{\varepsilon}{2^j} 
		&& \text{(Theorem 
			\ref{Additivityofvariationalmeasureoverintervals})}.
	\end{aligned}
	\]
	
	Since \(\sum_{j=1}^N \tfrac{\varepsilon}{2^j} \leq \varepsilon\), it 
	follows that
	\[
	\sum_{i=1}^m \rho(f(t_i))^p 1_A(t_i)(v_i-u_i) 
	\;\leq\; \mu(\rho (f)^p,A_N) + \varepsilon.
	\]
	
	Taking the supremum over all such \(\delta\)-fine partial partitions of 
	\([a,b]\), we obtain
	\[
	Var^p(\rho f,A,\delta) \;\leq\; \mu^p_\rho (f, A_N) + \varepsilon.
	\]
	Hence,
	\[
	\mu^p_\rho (f,A) 
	\;\leq\; \mu^p_\rho (f,A_N) + \varepsilon 
	\;\leq\; \lim_{n\to\infty} \mu^p_\rho (f,A_n) + \varepsilon.
	\]
	
	Since \(\varepsilon > 0\) is arbitrary, the result follows.
\end{proof}

\begin{theorem}[Fatou's Lemma]
	\label{Fatou'sLemma}
	Let \(\{f_n\}\) be a sequence in \(\mathcal{U}^p_{\rho}([a,b];X)\). Suppose \(\{f_n\}\) converges pointwisely a.e. to a function \(f : 
	[a,b]\to X\). Then
	\begin{equation}
		\label{Fatou'sLemmaInequality}
		\|f\|_{\mathcal{U}^p_{\rho}} \leq \liminf_{n \rightarrow 
			+\infty} \|f_n\|_{\mathcal{U}^p_{\rho}}.
	\end{equation}
\end{theorem}

\begin{proof}
	Without loss of generality, assume $f_n\to f$ pointwisely on $[a,b]$. Fix 
	$\eta>0$. Then there is a constant \(C(\eta)\) depending only on \(\eta\) such 
	that
	\begin{equation}
		\label{powerofsumestimatebypowers}
		(a+b)^p \le (1+\eta)\,a^p + C(\eta)\,b^p\qquad(a,b\ge 0).
	\end{equation}
	
	For $\varepsilon>0$, set
	\[
	E_n^\varepsilon:=\{x\in[a,b]:\rho(f(x)-f_n(x))<\varepsilon\}.
	\]
	Then $\{E_n^\varepsilon\}_n$ is an increasing sequence and, up to a null
	set,
	$[a,b]=\bigcup_{n=1}^\infty E_n^\varepsilon$. By 
	Theorem~\ref{measureofascendingsets} we have
	\begin{equation}\label{eq:Fatou-aux1}
		\mu^p_\rho(f,[a,b]) = \lim_{n \rightarrow \infty} 
		\mu^p_\rho(f,E_n^\varepsilon).
	\end{equation}
	
	Applying \eqref{powerofsumestimatebypowers} pointwise with
	$a=\rho(f_n(x))$ and $b=\rho(f(x)-f_n(x))$ gives, on $E_n^\varepsilon$,
	\[
	\rho(f(x))^p \le (1+\eta)\,\rho(f_n(x))^p + C(\eta)\,\varepsilon^p.
	\]
	Passing from the pointwise inequality to the variational measure (take sups over
	$\delta$-fine partitions and then inf over gauges) yields
	\[
	\mu^p_\rho(f,E_n^\varepsilon)
	\le (1+\eta) \mu^p_\rho(f_n,E_n^\varepsilon)
	+ C(\eta)\,\varepsilon^p\,m^*(E_n^\varepsilon).
	\]
	Using monotonicity of $\mu^p_\rho$ and
	$m^*(E_n^\varepsilon)\le b-a$ we obtain
	\begin{equation}\label{eq:Fatou-aux2}
		\mu(\rho(f)^{\,p},E_n^\varepsilon)
		\le (1+\eta)\,\mu(\rho(f_n)^{\,p},[a,b])
		+ C(\eta)\,\varepsilon^p (b-a).
	\end{equation}
	
	Now take $\liminf$ as $n\to\infty$ in \eqref{eq:Fatou-aux2}. Combining 
	with
	\eqref{eq:Fatou-aux1} gives
	\[
	\mu(\rho(f)^{\,p},[a,b])
	\le (1+\eta)\,\liminf_{n\to\infty}\mu(\rho(f_n)^p,[a,b])
	+ C(\eta)\,\varepsilon^p (b-a).
	\]
	Since $\varepsilon>0$ was arbitrary, we may let $\varepsilon\to 0$ to 
	obtain
	\[
	\mu(\rho(f)^p,[a,b])
	\le (1+\eta)\,\liminf_{n\to\infty}\mu(\rho(f_n)^p,[a,b]).
	\]
	Finally, let $\eta\downarrow 0$ to conclude
	\[
	\mu(\rho(f)^p,[a,b])
	\le \liminf_{n\to\infty}\mu(\rho(f_n)^p,[a,b]).
	\]
	Taking $p$-th roots and using the definition of the seminorm 
	$\|\cdot\|_{\mathcal U^p_\rho}$
	yields \eqref{Fatou'sLemmaInequality}.
\end{proof}

\begin{theorem}[Chebyshev's Inequality]
	\label{Chebeshev}
	Let \(f \in \mathcal{U}^p_{\rho}([a,b];X)\). Then for every \(a>0\),
	\[m^*(\{x \in [a,b]: \rho f(x)\geq a\}) \leq \frac{1}{a^p} 
	\|f\|_{\mathcal{U}^p_{\rho}}.\]
\end{theorem}
\begin{proof}
	For \(a>0\), let
	\[E_a=\{x \in [a,b]: \rho f(x)\geq a\}.\]
	Let \(\varepsilon>0\). Then there exists a gauge \(\delta\) on \([a,b]\)
	such that for \(\delta\)-fine partition \(D=\{[u_i,v_i]:1 \leq i \leq n\}\) of 
	\([a,b]\),
	\[\sup\{(D)\sum (\rho f(t))^p (v-u): D\text{ is }\delta\text{-fine}\} 
	< \|f\|_{\mathcal{U}^p_{\rho}} +\varepsilon.\]
	Now,
	\begin{align*}
		\sum_{i=1}^n a^p(v_i-u_i) &\leq \sum_{i=1}^n (\rho 
		f(t))^p(v_i-u_i)\\
		&\leq \sup\{(D)\sum (\rho f(t))^p (v-u): D\text{
			is }\delta\text{-fine}\}\\
		&< \|f\|_{\mathcal{U}^p_{\rho}} +\varepsilon
	\end{align*}
	it follows that
	\[a^p \sup\{(D)\sum 1_{E_a}(t)(v-u): D\text{ is }\delta\text{-fine}\} < \|f\|_{\mathcal{U}^p_{\rho}} +\varepsilon.\]
	Taking the infimum for all gauges \(\delta\) on \([a,b]\),
	\[a^p \cdot m^*(\{x \in [a,b]: \rho(f)\geq a\}) < 
	\|f\|_{\mathcal{U}^p_{\rho}} +\varepsilon.\]
	The result follows since \(\varepsilon\) was arbitrary.
\end{proof}

\section{Closure and completeness}
In this section, we discuss some notions of convergence and prove some 
completeness of \(L_\rho^p([a,b];X)\) and \(\mathcal{U}^p_{\rho}([a,b];X)\). We
begin by introducing and studying various definitions of convergence.
\begin{definition}
	\label{DefinitionofConvergences}
	Let \((X,d)\) be a pseudo-metric space, \(1 \leq p \leq \infty\) and let
	\(\{u_n\}\) be a sequence in \((X,d)\). Then,
	\begin{enumerate}
		\item \(\{u_n\}\) converges to \(u\) on \([a,b]\) if for every 
		\(\varepsilon>0\), there is \(N>0\) such that \(d(u_n,u)<\varepsilon\).
		\item \(\{u_n\}\) is Cauchy on \([a,b]\) if for every 
		\(\varepsilon>0\), there exists \(N>0\) such that \(d(u_m,u_n) < \varepsilon\) 
		for all \(m, n \geq N\).
		\item $\{u_n\}$ is \emph{rapidly Cauchy} if there exists a 
		convergent series
		$\sum_{k=1}^\infty\varepsilon_k$ of positive numbers such that
		\[
		d(u_{k+1},u_k) < \varepsilon_k^2\qquad\text{for all }k\in\mathbb
		N.\footnote{Here, we borrow the notion from \cite{Lax}.}
		\]
	\end{enumerate}
\end{definition}

In the literature (see, for example, \cite{Bartle,Royden,Rudin}), the proof of 
the completeness of normed spaces leverages the definition of absolute 
convergence. Consequently, the choice of a rapidly Cauchy sequence in proving 
the completeness of pseudo-metric spaces, which will be introduced later, is 
motivated by the absence of a definition of absolute convergence in these 
spaces.

\begin{remark}
	\label{CauchyandRapidlyCauchy}
	For a locally convex space \(X\), every subsequence of a Cauchy sequence
	\(\{x_n\}\) in \(X\) is also Cauchy. In addition, using the proof in 
	\cite{Royden} for normed spaces, every rapidly Cauchy sequence is Cauchy and 
	every Cauchy sequence has a subsequence that is rapidly Cauchy. 
\end{remark}

\begin{definition}
	A pseudo-metric space \((X,d)\) is said to be sequentially complete with
	respect to its pseudo-metric \(d\) if every Cauchy sequence in \(X\) converges 
	to some element of \(X\).
\end{definition}

\begin{definition}
	Let \((X,d)\) be a pseudo-metric space, \(A \subseteq \mathbb{R}\), 
	\(\{f_n:A \rightarrow X\}\) be a sequence of functions, and \(f:A \rightarrow X\) 
	be a given function. Then,
	\begin{enumerate}
		\item $\{f_n\}$ converges pointwise to $f$ on $A$ if the sequence 
		$\{f_n(t)\}$ converges to $f(t)$ for every $t\in A$.
		\item $\{f_n\}$ converges pointwise to $f$ a.e.\ on $A$ if there
		exists $E\subseteq A$ of measure zero such that $f_n\to f$ pointwise on
		$A\setminus E$.
		\item $\{f_n\}$ is pointwise Cauchy if for every $t\in[a,b]$ the
		sequence $\{f_n(t)\}$ is Cauchy.
		\item \(\{f_n\}\) is pointwise Cauchy a.e. on \(A\) if there 
		exists \(E \subseteq A\) of measure zero such that \(\{f_n\}\) is pointwise 
		Cauchy on \(A \setminus E\).
	\end{enumerate}
\end{definition}

One should notice that a semi-normed space \((X,\rho)\) is a pseudo-metric space
with pseudo-metric given by \(d(u,v) = \rho(u-v)\). Moreover, a Cauchy sequence
in a LCTVS that has a subsequence that is convergent must be convergent itself 
to the same limit. This may be shown in the following manner: Let \(\{x_k\}\) be 
a Cauchy sequence in a LCTVS \(X\) and let \(\{x_{k_i}\}\) be a subsequence 
that converges to \(x\). Let \(\varepsilon>0\) and \(\rho\) be a semi-norm in 
\(X\). Then there is \(N>0\) such that
\[\rho(x_m-x_n) < \frac{\varepsilon}{2} \, \text{for}\, m,n\geq N.\]
Also, there exists \(M>0\) such that
\[\rho(x_{k_i}-x)<\frac{\varepsilon}{2} \, \text{for}\, i \geq M.\]
Let \(k_j\) be large enough that \(k_j \geq N\). For \(n \geq N\),
\begin{align*}
	\rho(x_n-x) &\leq \rho(x_n-x_{k_j}) + \rho(x_{k_j} - x)\\
	&< \varepsilon.
\end{align*}
Since \(\rho\) was arbitrary, \(\{x_k\}\) converges to \(x\).

The proof below is adapted from \cite{Lax}.

\begin{theorem}
	\label{CompletenessofUpperIntegrableFunctioninSemiNormedSpaces}
	Let \((X,\rho)\) be a semi-normed space and \(1 \leq p \leq \infty\). 
	Then every rapidly Cauchy sequence in \(\mathcal{U}^p_{\rho}([a,b];X)\) is 
	pointwise Cauchy a.e. on \([a,b]\). Furthermore, if \(X\) is sequentially 
	complete, then \(\mathcal{U}^p_{\rho}([a,b];X)\) is sequentially complete.
\end{theorem}
\begin{proof}
	Let \(1 \leq p < +\infty\) and consider a rapidly Cauchy sequence 
	\(\{f_n\}\) in \(\mathcal{U}^p_{\rho}([a,b];X)\). Fix a convergent series
	\[\sum_{k=1}^\infty \varepsilon_k\]
	such that
	\begin{equation}
		\label{CompletenessofUpperIntegrablefunctions}
		\|f_{n+1} - f_n\|_{\mathcal{U}^p_{\rho}([a,b];X)}<\varepsilon_k^2 
		\quad \text{for all natural } k.
	\end{equation}
	that is,
	\[\mu^p_\rho\big(f_{n+1} - f_n,[a,b]\big) < \varepsilon_k^{2p}.\]
	Let
	\[E_n=\{x \in [a,b]:\rho(f_{n+1}(x)-f_n(x)) \geq \varepsilon_n\}.\]
	It follows then from Theorem \ref{Chebeshev} that for each \(n>0\),
	\begin{align*}
		m(E_n)
		&= m(\{x \in [a,b]:\rho(f_{n+1}(x)-f_n(x))  \geq 
		\varepsilon_n\})\\
		&< \frac{1}{\varepsilon_n^p} \mu^p_\rho\big(f_{n+1} - f_n,[a,b]\big)  \\
		&<\varepsilon_n^p.
	\end{align*}
	Now \(\sum_{n=1}^\infty \varepsilon_n\) converges since 
	\(\varepsilon_n^p \leq \varepsilon_n\). Applying the Borel-Cantelli Lemma to the
	sequence of measurable sets \(\{E_n\}\), there exists \(E_0\) of measure zero 
	such that for all \(x \in [a,b] \setminus E_0\) there exists \(N>0\) such that 
	for all \(n \geq N\),
	\[\rho(f_{n+1}(x) - f_n(x))<\varepsilon_n.\]
	Let \(\eta>0\). Then we can choose \(N_\eta \geq N\) large enough that
	\[\sum_{n=N_{\eta}}^\infty \varepsilon_n < \eta.\]
	Then, for all \(m \geq n \geq N\),
	\begin{align*}
		\rho(f_m(x)-f_n(x)) \leq \sum_{k=n}^{m-1} 
		\rho(f_{k+1}(x)-f_k(x)) <\sum_{k=n}^{m-1} \varepsilon_k \leq 
		\sum_{k=n}^{+\infty} \varepsilon_k<\eta.
	\end{align*}
	This shows that \(\{f_n(x)\}\) is Cauchy.
	
	Suppose further that \(X\) is sequentially complete. Then \(\{f_n(x)\}\)
	converges pointwise a.e. to a function \(f(x)\) on \([a,b]\). From 
	\eqref{CompletenessofUpperIntegrablefunctions},
	\begin{align*}
		\|f_{n+k} - f_n\|_{\mathcal{U}^p_{\rho}} &\leq \sum_{j=n}^\infty
		\varepsilon_j^2 \, \text{for all} \, n,k \geq 1.
	\end{align*}
	By Theorem \ref{Fatou'sLemma},
	\[\|f - f_n\|_{\mathcal{U}^p_{\rho}} \leq \liminf_{k \rightarrow 
		+\infty} \|f_{n+k} - f_n\|_{\mathcal{U}^p_{\rho}} \leq \sum_{j=n}^\infty 
	\varepsilon_j^2 \, \text{for all} \, n \geq 1.\]
	Hence, \(f \in \mathcal{U}^p_\rho([a,b];X)\). Since \(\sum_{j=1}^\infty 
	\varepsilon_j^2\) converges,
	\[\sum_{j=n}^\infty \varepsilon_j^2 \rightarrow 0.\]
	Thus, \(f_n \rightarrow f\) in \(\mathcal{U}^p_\rho([a,b];X)\). Therefore, 
	\(\mathcal{U}^p_\rho([a,b];X)\) is sequentially complete.
\end{proof}

\begin{theorem}
	\label{LpisClosedSubspace}
	Let \((X,\rho)\) be sequentially complete. Then the space 
	\(L_\rho^p([a,b];X)\) is a closed subspace of 
	\(\mathcal{U}^p_{\rho}([a,b];X)\).
\end{theorem}
\begin{proof}
	Let \(f\) be a limit point of \(L_\rho^p([a,b];X)\). We show that \(f 
	\in L_\rho^p([a,b];X)\) by constructing a \(\rho\)-SKH primitive. Suppose \(a 
	\leq t \leq b\).
	
	We show that \(f\) is Kurzweil-Henstock integrable on \([a,t]\). Since 
	\(f\) is a limit point, we can construct a sequence \(\{f_n\}_n \subseteq 
	L_\rho^p([a,t];X)\) such that
	\[\lim_{n \rightarrow +\infty}\|f-f_n\|_{\mathcal{U}^p_{\rho}([a,t];X)}
	= 0.\]
	Then the sequence is Cauchy in \(\mathcal{U}^p_{\rho}([a,t];X)\) and by 
	\eqref{eq:Up-embedding} of Lemma \ref{lem:Up-seminorm-and-embedding}, the sequence 
	\(\{f_n\}_n\) is Cauchy in \(\mathcal{U}^1_{\rho}([a,t];X)\).
	
	Let \(\varepsilon>0\). Then there exists \(N > 0\) such that
	\[\int_a^t \rho(f_m-f_n) =\mu^1_\rho(f_m - f_n,[a,t]) < \varepsilon \; 
	\text{for all} \; m,n \geq N.\]
	For \(k \geq 1\), let \(F_k\) be a \(\rho\)-SKH primitive of \(f_k\) so 
	that
	\[F_k(t) + \ker \rho = \int_a^t f_k.\]
	Proposition \ref{IntegralEstimateforp=1} implies that for all \(n,m 
	\geq N\), 
	\[\rho(F_m(t) - F_n(t)) \leq \int_a^t \rho(f_m-f_n) < \varepsilon.\]
	This shows that the sequence \(\{F_n(t)\}_n\) is a Cauchy sequence in 
	\(X\) and must converge to some class \(L + \ker \rho\). Define \(F(t) = L\) by 
	\textit{choosing} one of the limits of the sequence \(\{F_n(t)\}_n\).
	
	We show that \(F\) is a \(\rho\)-SKH primitive of \(f\). We begin by 
	showing that \(F(t)\) is a \(\rho\)-KH integral of \(f\). Let \(\varepsilon>0\).
	Then there exists \(N>0\) such that for all \(n \geq N\),
	\[\rho(F(t)-F_n(t))<\frac{\varepsilon}{4}\]
	and
	\[\|f-f_n\|_{\mathcal{U}_\rho^1([a,b];X)}<\frac{\varepsilon}{4}.\]
	Furthermore, for every \(\eta>0\) there is a gauge \(\delta_n\) on 
	\([a,b]\) such that for every \(\delta_n\)-fine partition \(D\) of \([a,b]\), we
	have
	\[\rho\left(F_n(t) - (D)\sum f_n(t)(v-u)\right) < \frac{\varepsilon}{4}\]
	and
	\[(D)\sum \rho (f_n(t)-f(t))(v-u) < 
	\|f-f_n\|_{\mathcal{U}_\rho^1([a,b];X)} +\frac{\varepsilon}{4} < \frac{\varepsilon}{2}.\]
	Then
	\[
	\begin{aligned}
		&\rho\left(F(t) - (D)\sum f(t)(v-u)\right)\\
		&\leq \rho(F(t) - F_n(t)) + \rho\left(F_n(t) - (D)\sum 
		f_n(t)(v-u)\right)\\
		&\quad + (D)\sum \rho (f_n(t)-f(t))(v-u)\\
		&<\varepsilon
	\end{aligned}
	\]
	for every \(\delta_n\)-fine partition \(D\) of \([a,b]\). This shows 
	that \(f\) is \(\rho\)-KH integrable with integral \(F(t)\). Applying 
	Proposition \ref{IntegralEstimateforforHKIntegral} and 
	\eqref{Additivityofintegral},
	\[
	\rho\left(F(v) - F(u)-(F_n(v) - F_n(u))\right) = \rho \left(\int_u^v f
	- \int_u^v f_n\right) \leq \mu^1_\rho(f-f_n,[u,v])
	\]
	for any \(a \leq u < v \leq b\).
	Thus, for any gauge \(\delta\) on \([a,b]\), if \(D\) is a 
	\(\delta\)-fine partial partition, then Lemma 
	\ref{Additivityofvariationalmeasureoverintervals} implies that
	\[
	\begin{aligned}
		(D) \sum \rho(\Delta F - \Delta F_n)([u,v],t) &\leq (D) \sum 
		\mu^1_\rho(f-f_n,[u,v])\\
		&= \mu^1_\rho(f-f_n,[a,b]).
	\end{aligned}
	\]
	It follows that
	\[
	\mu (\rho(\Delta F - \Delta F_n),[a,b]) \leq 
	\mu^1_\rho(f-f_n,[a,b]).
	\]
	
	Now, for all natural number \(n\), Remark \ref{SKHIntegrability} applies:
	\[
	\begin{aligned}
		\mu(\rho(\Delta F - \overline{f}),[a,b]) &\leq \mu (\rho(\Delta 
		F - \Delta F_n),[a,b]) + \mu (\rho(\Delta F_n - \overline{f_n}),[a,b]) \\
		&\quad + \mu^1_\rho(f-f_n,[a,b])\\
		&\leq 2 \mu^1_\rho(f-f_n,[a,b]).
	\end{aligned}
	\]
	Since
	\[\lim_{n \rightarrow +\infty}\|f-f_n\|_{\mathcal{U}^p_{\rho}([a,b];X)}
	= 0,\]
	then
	\[\mu(\rho(\Delta F - \overline{f}),[a,b]) = 0\]
	and by Remark \ref{SKHIntegrability}, \(f\) is \(\rho\)-SKH integrable. 
	Therefore, \(f \in L_\rho^p([a,b];X)\), showing that 
	\(L_\rho^p([a,b];X)\) is a closed subspace of \(\mathcal{U}_\rho^p([a,b];X)\).
\end{proof}

We can also prove this case for Fr\'echet spaces \cite{Rudin} with a slightly 
different argument.
\begin{lemma}
	\label{UniformRapidlyCauchy}
	Let \((X,\mathcal{P})\) be a Fr\'echet space with continuous semi-norms 
	\(\mathcal{P} = \{\rho_i:i \in \mathbb{N}\}\), that is, a locally convex space 
	with a topology that can be generated by countable collection continuous semi-
	norms. Then every Cauchy sequence \(\{u_n\}\) in \(X\) contains a subsequence 
	\(\{f_{n_k}\}\) such that there is a convergent series
	\[\sum_{k=1}^\infty \varepsilon_k\]
	with the property that for all \(j \geq 1\),
	\[\rho_j (f^{n_{k+1}} - f^{n_k}) < \varepsilon^2_k \; \text{for} \; k 
	\geq j.\]
\end{lemma}
\begin{proof}
	We will exploit the fact that every subsequence of a Cauchy sequence in 
	\((X,\mathcal{P})\) is also Cauchy in \((X,\mathcal{P})\) and a diagonalization 
	process. For this proof, we write the index as a superscript. Let \(\{f^i\}\) be
	a given Cauchy sequence. We show that there a double sequence \(\{n^l_k\}\) of natural numbers such that for all \(l \geq 1\), if \(1 \leq j \leq l\) and \(k \geq 1\),
	\begin{equation}
		\label{RapidCauchyFrechet1}
		\rho_j \left(f^{n^l_{k+i}} - f^{n^l_k} \right) <\frac{1}{2^k} \, 
		\text{for} \, i \geq 1.
	\end{equation}

	Now, there exists a subsequence \(\{n^1_i\}\) such
	that for \(k \in \mathbb{N}\),
	\[\rho_1(f^{n^1_{k+i}} - f^{n^1_k}) < \frac{1}{2^k} \; \text{for all} \,
	i \geq 1.\]
	Moreover, \(\{f^{n^1_i}\}\) is a subsequence and so, it should also be Cauchy. In particular, there exists a subsequence \(\{f^{n^2_i}\}\) of \(\{f^{n^1_i}\}\) such that for \(k \in \mathbb{N}\),
	\[\rho_2(f^{n^2_{k+i}} - f^{n^2_k}) < \frac{1}{2^k} \; \text{for all} \, i \geq 1\]
	and
	\[\rho_1(f^{n^2_{k+i}} - f^{n^2_k}) < \frac{1}{2^k} \; \text{for all} \, i \geq 1.\]
	
	Suppose that for a natural number \(l\) we have constructed the sequence
	\(\{f^{n^l_{k}}\}_k\) satisfying \eqref{RapidCauchyFrechet1} for \(1 \leq j \leq l\). Then \(\{f^{n^l_k}\}\) is a Cauchy sequence in \(\rho_{l+1}\) since it is a subsequence of the Cauchy sequence \(f^n\) in \((X,\mathcal{P})\).
	
	For each \(k \geq 1\), there is a number \(N(l+1,k)>k\) such that
	\[\rho_{l+1}(f^{n^l_{p}} - f^{n^l_q}) < \frac{1}{2^k} \, \text{for all} \, p,q \geq N(l+1,k).\]
	We may choose \(N(l+1,k)\) to be increasing in \(k\). Then we let \(n^{l+1}_k = 
	n^l_{N(l+1,k)}\) for \(k\geq 1\). It follows that
	\[\rho_{l+1} \left(f^{n^{l+1}_{k+i}} - 
	f^{n^{l+1}_k}\right)<\frac{1}{2^k} \; \text{for} \; i \geq 1.\]
	Hence, for \(1 \leq j \leq l\), we have \(N(l+1,k) > N(l+1,k+i) >k\) and that
	\[\rho_j \left(f^{n^{l+1}_{k+i}} - f^{n^{l+1}_k} \right) = \rho_j 
	\left(f^{n^l_{N(l+1,k+i)}} - f^{n^l_{N(l+1,k)}}\right) <\frac{1}{2^k} \, 
	\text{for} \, i \geq 1.\]
	This proves the claim.
	
	Consider the subsequence \(\{f^{m_k}\} = \{f^{n^k_k}\}\). For \(k \geq 
	j\) and \(i \geq 1\), \(n^{k+i}_{k+i} = n^k_{P}\) for some \(P>k\).  It follows 
	that
	\[\rho_j (f^{n^{k+i}_{k+i}} - f^{n^k_k}) = \rho_j (f^{n^k_P} - 
	f^{n^k_k}) < \frac{1}{2^k}.\]
	In other words,
	\[\rho_j (f^{m_{k+i}} - f^{m_k}) < \frac{1}{2^k} \; \text{for} \; 1 
	\leq j \leq k \; \text{and} \; i\geq 1.\]
	Now that \(\varepsilon_k = \frac{1}{\sqrt{2^k}}<1\), the geometric series 
	\(\sum_{k=1}^\infty \varepsilon_k\) converges such that
	\[\rho_j (f^{m_{k+1}} - f^{m_k}) < \varepsilon_k^2 \; \text{for} \; 1 
	\leq j \leq k.\]
	This proves the lemma.
\end{proof}

The result above motivates a definition of rapidly Cauchy sequence for Fr\'echet spaces.

\begin{definition}
	Let \(X\) ne a Fr\'echet space. A sequence \(\{u_k\}\) in \(X\) is said to be a rapidly Cauchy if there exists a convergent series
	\[\sum_{k=1}^\infty \varepsilon_k\]
	with the property that for all \(j \geq 1\),
	\[\rho_j (f^{n_{k+1}} - f^{n_k}) < \varepsilon^2_k \; \text{for} \; k 
	\geq j.\]
\end{definition}

\begin{remark}
	For a given Fr\`echet space \((X,\mathcal{P})\) and \(1 \leq p < \infty\),
	the space \(\mathcal{U}^p([a,b];X)\) is the intersection of all the 
	spaces \(\mathcal{U}^p_{\rho}([a,b];X)\) for all \(\rho \in \mathcal{P}\). It follows that a rapidly Cauchy sequence in \(X\) is a rapidly Cauchy sequence in \((X,\rho)\) for any continuous semi-norm \(\rho\) on \(X\).
\end{remark}

\begin{theorem}
	\label{AUHKisComplete}
	Let \(X\) be a sequentially complete Fr\'echet space and \(1 \leq p < 
	+\infty\). The space \(\mathcal{U}^p([a,b];X)\) is a sequentially complete 
	Fr\'echet space.
\end{theorem}
\begin{proof}
	Note first that \(\{\|\cdot \|_{\mathcal{U}^p_{\rho_i}}\}_i\) forms a countable collection of semi-norms on 
	\(\mathcal{U}^p([a,b];X)\) where \( \mathcal{P} = \{\rho_i: i \in \mathbb{N}\}\) is the countable separating family of continuous semi-norms on \(X\). It remains to show that \(\mathcal{U}^p([a,b];X)\) 
	is sequentially complete.
	
	Let \(\{f_n\}\) be a Cauchy sequence in 
	\(\mathcal{U}^p([a,b];X)\). Then there is a subsequence, still denoted by
	\(n\), and a convergent series
	\[\sum_{n=1}^\infty \varepsilon_n\]
	such that for all \(i \geq 1\),
	\[\rho_i (f_{n_{k+1}} - f_{n_k}) < \varepsilon^2_k \; \text{for} \; 1 
	\leq i \leq k.\]
	We may proceed similar to the proof of Theorem 
	\ref{CompletenessofUpperIntegrableFunctioninSemiNormedSpaces} to show 
	that for each \(i\), there is \(E_{\rho_i} \subset [a,b]\) of measure zero such 
	that for all \(x \in [a,b] \setminus E_{\rho_i}\), \(\{f_n(x)\}\) is Cauchy in 
	\((X,\rho_i)\). Let \(E_0\) be the union of all the sets \(E_{\rho_i}\). Then 
	\(E_0\) has measure zero and for all \(x \in [a,b] \setminus E_0\) and \(i \geq 
	1\), \(\{f_n(x)\}\) is Cauchy in \((X,\rho_i)\) and so, \(\{f_n(x)\}\) is Cauchy
	in \((X,\mathcal{P})\). Since \((X,\mathcal{P})\) is sequentially complete, 
	there is a function \(f:[a,b] \rightarrow X\) which is a pointwise limit a.e. on
	\([a,b]\) of the sequence \(\{f_n\}\). In particular, \(f\) is a pointwise 
	limit of \(\{f_n\}\) in \((X,\rho_i)\) for each \(i\). Applying Theorem 
	\ref{Fatou'sLemma},
	\[\|f\|_{\mathcal{U}^p_{\rho_i}} < +\infty \; \text{for all} \; i \geq 1.\]
	Since a subsequence of \(\{f_n\}\) converges to \(f\) then \(\{f_n\}\) 
	also converges to \(f\).
\end{proof}

The next result follows from Theorems \ref{AUHKisComplete} and 
\ref{LpisClosedSubspace}.

\begin{corollary}
	Let \(X\) be a sequentially complete Fr\'echet space and \(1 \leq p < 
	+\infty\). The space \(L^p([a,b];X)\) is a closed subspace of 
	\(\mathcal{U}^p([a,b];X)\). Furthermore, \(L^p([a,b];X)\) is a sequentially 
	complete Fr\'echet space.
\end{corollary}

\section{Conclusion}

In this paper, we developed the framework of $\rho$-upper integrability for 
functions with values in semi-normed and locally convex topological vector 
spaces. By introducing the spaces $\mathcal{U}^p_{\rho}$ and their associated 
semi-norms, we established a foundation that extends classical notions of 
absolute integrability to a broader functional-analytic setting.

We proved the sequential completeness of $\mathcal{U}^p_{\rho}([a,b];X)$ 
whenever the target space $X$ is sequentially complete, and demonstrated that 
the absolutely integrable subspace $L^p_{\rho}([a,b];X)$ is closed inside 
$\mathcal{U}^p_{\rho}([a,b];X)$.

The techniques presented here suggest several directions for further study. 
Among them are the extension of $\rho$-upper integrability to broader classes 
of vector-valued functions, the investigation of duality and operator theory 
in this setting, and possible applications to integration in infinite-dimensional 
spaces. Such developments may help unify approaches to non-classical integrals 
within a coherent functional-analytic theory.

\backmatter

%\bmhead{Acknowledgements}
%
%Acknowledgements are not compulsory. Where included they should be brief. Grant or contribution numbers may be acknowledged.
%
%Please refer to Journal-level guidance for any specific requirements.

\section*{Declarations}

The author declare that no funds, grants, or other support were received during the preparation of this manuscript.

\bibliography{sn-bibliography}% common bib file

\end{document}